\documentclass[11pt]{article}
\usepackage{a4}
\usepackage{tikz}
\usepackage{latexsym}
\usepackage{amsfonts}
\usepackage{amssymb}
\usepackage{amsmath}
\usepackage{amsthm}
\usepackage{float}
\newtheorem{theorem}{Theorem}[section]
\newtheorem{lemma}[theorem]{Lemma}
\newtheorem{corollary}[theorem]{Corollary}
\newtheorem{conjecture}[theorem]{Conjecture}

\newtheorem{claim}[theorem]{Claim}
\title{New sum-product estimates for real and complex numbers}
\author{Antal Balog\footnote{The first author was partially supported by HNSF Grant K104183 and K109789.} \,and Oliver Roche-Newton\footnote{The second author was partially supported by Grant ERC-AdG. 321104 and EPSRC Doctoral Prize Scheme (Grant Ref:  EP/K503125/1)}}

\begin{document}
\maketitle

\begin{abstract} A variation on the sum-product problem seeks to show that a set which is defined by additive and multiplicative operations will always be large. In this paper, we prove new results of this type. In particular, we show that for any finite set $A$ of positive real numbers, it is true that
$$\left|\left\{\frac{a+b}{c+d}:a,b,c,d\in{A}\right\}\right|\geq{2|A|^2-1}.$$
As a consequence of this result, it is also established that
$$|4^{k-1}A^{(k)}|:=|\underbrace{\underbrace{A\cdots{A}}_\textrm{k times}+\cdots{+A\cdots{A}}}_\textrm{$4^{k-1}$ times}|\geq{|A|^k}.$$
Later on, it is shown that both of these bounds hold in the case when $A$ is a finite set of complex numbers, although with smaller multiplicative constants.
\end{abstract}

\section{Introduction}

Given a pair of finite sets $A,B\subset{\mathbb{R}}$, we define their sum set, and respectively their product set, to be the sets
$$A+B:=\{a+b:a\in{A},b\in{B}\}$$
and
$$AB:=\{ab:a\in{A},b\in{B}\}.$$
More generally, for a given integer $k\geq{2}$, we can define the $k$-fold sum set and $k$-fold product set of $A$ to be the sets
$$kA:=\{a_1+\cdots+a_k:a_1,\cdots,a_k\in{A}\}$$
and
$$A^{(k)}:=\{a_1\cdots{a_k}:a_1,\cdots,a_k\in{A}\},$$
respectively. The Erd\H{o}s-Szemer\'{e}di sum-product conjecture \cite{ES} states that, for every finite set $A\subset{\mathbb{Z}}$, we have
$$\max{\{|A+A|,|AA|\}}\gg{|A|^{2-\epsilon}},\,\,\,\,\,\,\,\,\,\,\,\,\,\forall{\epsilon>0}.$$
They also gave a conjecture for $k$-fold sum and product sets which claimed that
$$\max{\{|kA|,|A^{(k)}|\}}\gg{|A|^{k-\epsilon}},\,\,\,\,\,\,\,\,\,\,\,\,\,\,\,\,\forall{\epsilon>0}.$$
Although these conjectures were originally stated for integers, it is natural to extend the problem to the case in which $A$ is a set of real numbers. The case of $n=2$ is the most famous one, and for the reals, and indeed integers, the best result comes from a beautiful argument of Solymosi \cite{solymosi}, who used elementary geometry to prove that
$$\max{\{|A+A|,|AA|\}}\geq{\frac{|A|^{4/3}}{2\lceil{\log{|A|}}\rceil}},$$
for a set $A$ of positive real numbers.

In a similar spirit to the Erd\H{o}s-Szemer\'{e}di conjecture, one might expect that given a set $A$ of numbers, a set formed by a combination of additive and multiplicative operations applied to elements of $A$ will grow significantly. Some progress has been made in this direction; for instance, in \cite{balog} the first author proved that
\begin{equation}
|AA+AA+AA+AA|\geq{\frac{1}{2}|A|^2},
\label{4AA}
\end{equation}
for a set $A$ of positive reals, using arguments similar to those of Solymosi \cite{solymosi}. In \cite{rectangles}, the second author and Misha Rudnev used a more involved geometric argument, building on the work of Guth and Katz \cite{GK} on the Erd\H{o}s distance problem, in order to prove that
\begin{equation}
|(A+A)(A+A)|\gg{\frac{|A|^2}{\log{|A|}}},
\label{rect}
\end{equation}
for any $A\subset{\mathbb{R}}$. In fact, \eqref{rect} follows as a consequence of the following result\footnote{The bound in \eqref{solutions} is the content of Proposition 4 in \cite{rectangles}, and the bulk of that paper is devoted to proving this.}: the number of solutions to the equation
\begin{equation}
(a_1+a_2)(a_3+a_4)=(a_5+a_6)(a_7+a_8),
\label{solutions}
\end{equation}
such that $a_i\in{A}$, is $O(|A|^6\log{|A|})$. Once we have an effective upper bound for the number of solutions to \eqref{solutions}, it requires only a simple application of the Cauchy-Schwarz inequality to complete the proof of \eqref{rect}. In fact, since the number of solutions to \eqref{solutions} is exactly the same\footnote{We can avoid any potential issues arising from division by zero by simply assuming that $A$ is a set of positive reals, and this will only change the implied constant in the eventual result.} as the number of solutions to
$$\frac{a_1+a_2}{a_5+a_6}=\frac{a_7+a_8}{a_3+a_4},$$
an application of Cauchy-Schwarz, twinned with the aforementioned bound on the number of solutions to \eqref{solutions}, is also sufficient to conclude that
\begin{equation}
\left|\frac{A+A}{A+A}\right|\gg{\frac{|A|^2}{\log{|A|}}}.
\label{weak}
\end{equation}

The first new result in this paper provides an improvement on \eqref{weak}, by showing that
\begin{equation}
\left|\frac{A+A}{A+A}\right|\geq{2|A|^2-1},
\label{mainresult1}
\end{equation}
holds for a finite set $A$ of positive reals. As well as giving a quantitative improvement on \eqref{weak} by removing a log factor and giving nice constants, the proof of \eqref{mainresult1} is elementary, and is inspired by the geometric argument of Solymosi \cite{solymosi}. We will see later that advantage of the simplicity of this approach is that we can extend the result to hold for complex numbers. On the other hand, it is not known at present whether the methods established by Guth and Katz extend to the complex setting.

We remark that this bound is optimal, up to the multiplicative constant, as is illustrated by the case when $A=\{1,\cdots,|A|\}$. Indeed, for some small examples, this bound is completely tight, and it can be checked that if $A=\{1,2,3\}$ then 
$$\left|\frac{A+A}{A+A}\right|=17=2|A|^2-1.$$
If $A=\{1,\cdots,|A|\}$, and $|A|$ is much larger, then the bound in Theorem \ref{main1} is also close to being tight. Indeed, if $A$ takes this form, then
$$A+A=\{2,\cdots,2|A|\}\subset{\{1,\cdots,2|A|\}}.$$
This means that
$$\frac{A+A}{A+A}\subset{\left\{\frac{n}{m}:n,m\in{\mathbb{Z}},1\leq{n,m}\leq{2|A|}\right\}}.$$
The size of the set on the right-hand side is precisely the number of pairs of coprime integers $(n,m)$ such $1\leq{n,m}\leq{2|A|}$, which tends to $(2|A|)^2\frac{6}{\pi^2}$ as $|A|$ tends to infinity. Therefore, we can construct an arbitrarily large set $A$ such that
$$\left|\frac{A+A}{A+A}\right|<\frac{24}{\pi^2}|A|^2\approx{2.4317|A|^2}.$$
On the other hand, the logarithmic factor cannot be completely eliminated from \eqref{rect}; if we again take $A=\{1,\cdots,|A|\}$, then
$$(A+A)(A+A)\subset{\{1,\cdots,2|A|\}\{1,\cdots,2|A|\}},$$
and it is known that the product set of the first $2|A|$ integers has cardinality $o(|A|^2)$. See Ford \cite{ford} for a more accurate statement regarding the size of this product set. It is interesting to note that these two similar looking problems - finding the best possible lower bounds for $|(A+A)(A+A)|$ and $\left|\frac{A+A}{A+A}\right|$ - have subtle differences. For further comparison, we point out that a geometric result of Ungar \cite{ungar} concerning the number of distinct directions determined by a planar point set yields the bound
$$\left|\frac{A-A}{A-A}\right|\geq{|A|^2-1}$$
as a corollary.

As an application of our bound for the size of $\frac{A+A}{A+A}$, we combine this with a sum-product type lemma from \cite{balog}, and use an inductive argument to show that
\begin{equation}
|4^{k-1}A^{(k)}|\geq{|A|^k},
\label{mainresult2}
\end{equation}
for a set of positive reals $A$, and any positive integer $k$. Note that the case when $k=2$ gives an improved constant in the bound \eqref{4AA}. As was noted in \cite{balog}, this bound is close to optimal. For instance, if $A=\{1,\cdots,N\}$, then $AA\subset{\{1,\cdots,N^2\}}$, which implies that
$$AA+AA+AA+AA\subset{\{4,\cdots,4N^2\}},$$
and so $|AA+AA+AA+AA|<4|A|^2$. One might compare the bound \eqref{mainresult2} with the work of Bourgain and Chang \cite{BC} on $k$-fold sum-product estimates. It was established in \cite{BC} that given an integer $b$, there exists $k=k(b)\in{\mathbb{Z}}$ such that
$$\max{\{|kA|,|A^{(k)}|\}}>|A|^b$$
holds for any set $A\subset{\mathbb{Z}}$. One can take $k(b)=C^{b^4}$, for some absolute constant $C$.

Recently, Konyagin and Rudnev \cite{KR} extended Solymosi's sum-product estimate, up to multiplicative constants, to the case whereby $A\subset{\mathbb{C}}$. Since \eqref{mainresult1} and \eqref{mainresult2} are proved here by a geometric argument which is similar to Solymosi's, the work of Konyagin and Rudnev lays a foundation for these results to be extended to the complex setting as well. It turns out that proving a version of \eqref{mainresult1} for complex numbers is relatively straightforward, but generalising \eqref{mainresult2} requires more care. We consider these problems for the reals in section 2, and the complex setting is covered in section 3.

\section{The Real Setting}

\begin{theorem} \label{main1} If $A$ is a finite set of positive reals, then
$$\left|\frac{A+A}{A+A}\right|\geq{2|A|^2-1}.$$
\end{theorem}

\begin{proof}

Following the setup of Solymosi, consider the point set $P=A\times{A}$, which lies in the positive quadrant of the plane $\mathbb{R}^2$. Observe that the sum set $P+P$ has the property that
\begin{equation}
P+P=(A\times{A})+(A\times{A})=(A+A)\times{(A+A)}.
\label{fact1}
\end{equation}
One can cover the point set $P$ by lines through the origin, as shown in Figure 1. These lines are labelled $l_1,\cdots,l_k$, in increasing order of gradient. To be more precise, the line $l_i$ has equation $y=m_ix$, and $m_1<\cdots<m_k$. 

Since each point from $P$ lies on exactly one of these $k$ lines, it follows that
\begin{equation}
\sum_{i=1}^k|l_i\cap{P}|=|P|=|A|^2.
\label{obvious}
\end{equation}

Notice that the number $k$ of lines through the origin covering $P$ is equal to the cardinality of the ratio set $A/A$. This is because the set of $\{m_1,\cdots,m_k\}$ of slopes is precisely the ratio set $A/A$. By the same reasoning, the quantity we are interested in, the cardinality $\left|\frac{A+A}{A+A}\right|$, is precisely the number of lines through the origin needed to cover the point set $(A+A)\times{(A+A)}$. 

Therefore, by \eqref{fact1}, we have reduced the problem of bounding $\left|\frac{A+A}{A+A}\right|$ from below, to the geometric problem of showing that $P+P$ determines\footnote{The number of lines $P+P$ \textit{determines} is just the number of lines through the origin which are needed in order to cover $P+P$.} many lines through the origin. The aim is now to show that $P+P$ determines at least $2|A|^2-1$ such lines.

To this end, we begin by recalling a geometric observation of Solymosi: for any $1\leq{i}\leq{k-1}$, and given two points $p=(p_1,p_2)\in{l_i\cap{P}}$ and $q=(q_1,q_2)\in{l_{i+1}\cap{P}}$, the vector sum $p+q$ lies strictly inside the region of the positive quadrant in between $l_i$ and $l_{i+1}$. This is simply because

\begin{equation}
\frac{p_2}{p_1}<\frac{q_2}{q_1}\Rightarrow{\frac{p_2}{p_1}<\frac{p_2+q_2}{p_1+q_1}<\frac{q_2}{q_1}}.
\label{segment}
\end{equation}

Given a point $p=(p_1,p_2)$ in the plane, the notation $R(p)$ will now be used to denote the gradient of the line connecting the origin and $p$, so that
\begin{equation}
R(p):=\frac{p_2}{p_1}.
\label{Rdef}
\end{equation}

Next, we make the important observation that, if $p\in{l_i\cap{P}}$ and $q,q'\in{l_{i+1}\cap{P}}$ with magnitude $|q'|>|q|$, then

\begin{equation}
R(p+q)<R(p+q').
\label{fact2}
\end{equation}

The best explanation for why this is true comes from studying Figure 1, but a purely algebraic proof can also be given. Indeed, given these conditions, one can check that
$$\frac{q_2'-q_2}{q_1'-q_1}=\frac{q_2'}{q_1'},$$
and therefore
\begin{align}
\frac{p_2+q_2}{p_1+q_1}<\frac{p_2+q_2'}{p_1+q_1'}&\Leftrightarrow{p_1q_2+q_1'p_2<p_1q_2'+q_1p_2}
\\&\Leftrightarrow{p_2(q_1'-q_1)<p_1(q_2'-q_2)}
\\&\Leftrightarrow{\frac{p_2}{p_1}<\frac{q_2'-q_2}{q_1'-q_1}=\frac{q_2'}{q_1'}}. \label{last}
\end{align}
The hypothesis that $(p_1,p_2)\in{l_i}$ and $(q_1,q_2)\in{l_{i+1}}$ implies that $\frac{p_2}{p_1}<\frac{q_2'}{q_1'}$, so that \eqref{last} holds, and this verifies claim \eqref{fact2}. The same argument shows that, if $q\in{l_{i+1}}$ and $p,p'\in{l_i}$ with $|p'|>|p|$, then
\begin{equation}
R(p+q)>R(p'+q).
\label{fact3}
\end{equation}
Note that it is at this point that we make use of the condition in the theorem that $A$ contains only positive elements. In the step leading immediately to \eqref{last}, we do not need to reverse the direction of the inequality when dividing by $p_1$, since the fact that $p\in{P}$ means that $p_1\in{A}$, which in turn implies that $p_1>0$. This condition is used similarly to establish \eqref{segment}.

Next, let us focus on the number of lines determined by $P+P$ whose slope belongs to the open interval $(m_i,m_{i+1})$. In other words, let us focus on the elements of $\frac{A+A}{A+A}$ inside this interval. The following argument, illustrated by Figure 1, will show that at least
$$|l_i\cap{P}|+|l_{i+1}\cap{P}|-1$$
lines through the origin determined by $P+P$ have a slope inside $(m_i,m_{i+1})$.

To do this, it is necessary to introduce more notation. Label the points in $l_i\cap{P}$ as $p_i^{(1)},\cdots,p_i^{(n_i)}$, where $|p_i^{(1)}|<\cdots<|p_i^{(n_i)}|$ and
$$n_i:=|l_i\cap{P}|.$$
It follows from \eqref{fact2} and \eqref{fact3} that
\begin{align*}
R\left(p_i^{(1)}+p_{i+1}^{(n_{i+1})}\right)&>R\left(p_i^{(1)}+p_{i+1}^{(n_{i+1}-1)}\right)
\\&>R\left(p_i^{(1)}+p_{i+1}^{(n_{i+1}-2)}\right)
\\&>\cdots
\\&>R\left(p_i^{(1)}+p_{i+1}^{(1)}\right)
\\&>R\left(p_i^{(2)}+p_{i+1}^{(1)}\right)
\\&>\cdots
\\&>R\left(p_i^{(n_i)}+p_{i+1}^{(1)}\right).
\end{align*}
This argument identifies $n_{i+1}+n_i-1$ elements of $P+P$, all of which determine distinct lines through the origin with slope in the open interval $(m_i,m_{i+1})$. This in turn identifies $n_{i+1}+n_i-1$ elements of the set $\frac{A+A}{A+A}$ lying in the interval $(m_i,m_{i+1})$. Summing over all $i$, it follows that

\begin{align}
\left|\frac{A+A}{A+A}\right|&\geq{\sum_{i=1}^{k-1}(n_{i+1}+n_i-1)}
\\&=2\left(\sum_{i=1}^{k}n_i\right)-n_1-n_k-(k-1)
\\&=2|A|^2-n_1-n_k-k+1 \label{eq1},
\end{align}
where \eqref{eq1} is a consequence of \eqref{obvious}. 

We observe that $n_1=n_k=1$. This is because $n_1$, for example, is the number of points in $l_1\cap{P}$; that is, the number of points from $P$ lying on the line with the most shallow gradient in our covering set of lines. The only point from $P$ on this line is the point $(a_{max},a_{min})$, where $a_{max}$ and $a_{min}$ are the largest and smallest elements of $A$ respectively. In Figure 1, this is the point in the bottom right corner of our direct product point set $P$. Similarly, the only point in $l_{k}\cap{P}$ is $(a_{min},a_{max})$. Putting this information into \eqref{eq1}, we now have

\begin{equation}
\left|\frac{A+A}{A+A}\right|\geq{2|A|^2-k-1}.
\label{almost}
\end{equation}

However, in obtaining bound \eqref{almost}, we have only taken into account elements of $\frac{A+A}{A+A}$ in the open intervals $(m_1,m_2),(m_2,m_3),\cdots,(m_{k-1},m_k)$, and not the values of $m_i$ themselves. It can quickly be verified that for all $1\leq{i}\leq{k}$, it is true that 
$$m_i\in{\frac{A+A}{A+A}}.$$ 
Indeed, for any such $i$, there exists at least one point $p=(p_1,p_2)\in{A\times{A}}$ such that $\frac{p_2}{p_1}=m_i$. Then, 
$$\frac{p_2+p_2}{p_1+p_1}=\frac{p_2}{p_1}=m_i.$$
In this way, we have identified a further $k$ elements of $\frac{A+A}{A+A}$, which were not counted in the analysis which led to bound \eqref{almost}. Therefore, we can improve slightly on \eqref{almost} in establishing that
$$\left|\frac{A+A}{A+A}\right|\geq{2|A|^2-k-1+k}=2|A|^2-1,$$
which completes the proof.
\end{proof}

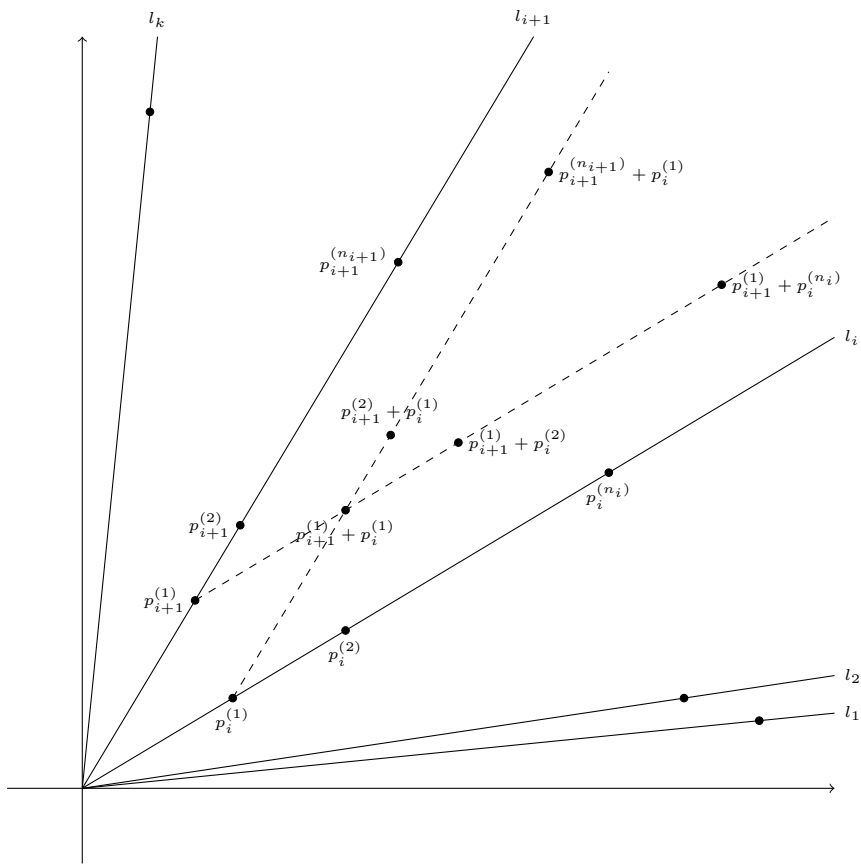
\begin{figure}[H]
\centering
\begin{tikzpicture}[font=\tiny, scale=1]

\draw [->] (0,-1) -- (0,10);
\draw [->] (-1,0) -- (10,0);
\draw (0,0) -- (10,1);
\draw (0,0) -- (10,1.5);
\draw (0,0) -- (10,6);
\draw (0,0) -- (6,10);
\draw (0,0) -- (1,10);
\draw [dashed] (2,1.2) -- (7,9.53333);
\draw [dashed] (1.5,2.5) -- (10,7.6);
\draw[fill] (9,0.9) circle [radius=0.05];
\draw[fill] (8,1.2) circle [radius=0.05];
\draw[fill] (2,1.2) circle [radius=0.05];
\draw[fill] (3.5,2.1) circle [radius=0.05];
\draw[fill] (7,4.2) circle [radius=0.05];
\draw[fill] (1.5,2.5) circle [radius=0.05];
\draw[fill] (2.1,3.5) circle [radius=0.05];
\draw[fill] (4.2,7) circle [radius=0.05];
\draw[fill] (0.9,9) circle [radius=0.05];
\draw[fill] (3.5,3.7) circle [radius=0.05];
\draw[fill] (4.1,4.7) circle [radius=0.05];
\draw[fill] (6.2,8.2) circle [radius=0.05];
\draw[fill] (5,4.6) circle [radius=0.05];
\draw[fill] (8.5,6.7) circle [radius=0.05];
\node[right] at (10,1) {$l_1$};
\node[right] at (10,1.5) {$l_2$};
\node[right] at (10,6) {$l_{i}$};
\node[above] at (6,10) {$l_{i+1}$};
\node[above] at (1,10) {$l_k$};
\node[below] at (2,1.2) {$p_i^{(1)}$};
\node[below] at (3.5,2.1) {$p_i^{(2)}$};
\node[below] at (7,4.2) {$p_i^{(n_i)}$};
\node[left] at (1.5,2.5) {$p_{i+1}^{(1)}$};
\node[left] at (2.1,3.5) {$p_{i+1}^{(2)}$};
\node[left] at (4.2,7) {$p_{i+1}^{(n_{i+1})}$};
\node[below] at (3.5,3.7) {$p_{i+1}^{(1)}+p_{i}^{(1)}$};
\node[above] at (4.1,4.7) {$p_{i+1}^{(2)}+p_{i}^{(1)}$};
\node[right] at (6.2,8.2) {$p_{i+1}^{(n_{i+1})}+p_{i}^{(1)}$};
\node[right] at (5,4.6) {$p_{i+1}^{(1)}+p_{i}^{(2)}$};
\node[right] at (8.5,6.7) {$p_{i+1}^{(1)}+p_{i}^{(n_i)}$};
\end{tikzpicture}

\caption{Illustration of the proof of Theorem \ref{main1}}

\centering

\end{figure}

Using the notation introduced in \eqref{Rdef}, we can define the set of slopes determined by a point set $P\subset{\mathbb{R}^2}$ by the notation
$$R(P):=\{R(p):p\in{P}\}.$$
In this language Theorem \ref{main1} tells us that, if $P=A\times{A}$, then 
\begin{equation}
|R(P+P)|\geq{2|P|-1}.
\label{general}
\end{equation}
It is not difficult to generalise Theorem \ref{main1} so that it holds for a point set $P$ which is not necessarily a direct product, in the form of the following result:
\begin{theorem} \label{generalise}
Let $P\subset{\mathbb{R}^2}$ be a finite point set which is not contained on a single line through the origin. Then
$$|R(P+P)|\geq{|P|+1}.$$
\end{theorem}
\begin{proof}
Since the proof is very similar to that of Theorem \ref{main1}, we will give only a brief sketch. Let us assume for simplicity that all the points of $P$ are contained in the positive quadrant. Cover $P$ by a set of lines $L$, let $l_1$ the line with smallest gradient, and let $l_k$ be the steepest. Taking vector sums of points from $P$ along neighbouring lines as before gives
$$|R(P)|\geq{k+\sum_{i=1}^{k-1}(|l_i\cap{P}|+|l_{i+1}\cap{P}|-1)}=2|P|+1-|l_1\cap{P}|-|l_k\cap{P}|.$$
Since $P$ is not contained on a single line through the origin, we know that $l_1\neq{l_k}$, and therefore $|l_1\cap{P}|+|l_k\cap{P}|\leq{|P|}$. This completes the proof.
\end{proof}
Observe that this bound is optimal, as is illustrated by the case when $P$ consists of $|P|-1$ points on a line through the origin, with a single point lying away from the line.

Now we begin to move towards a proof of Theorem \ref{main2}. Apart from Theorem \ref{main1}, the other main ingredient is the following lemma from \cite{balog}:

\begin{lemma} \label{baloglemma} Let $\mathcal{A,B,C}$ and $\mathcal{D}$ be finite sets of positive real numbers. Then
$$|\mathcal{AC}+\mathcal{AD}||\mathcal{BC}+\mathcal{BD}|\geq{|\mathcal{A}/\mathcal{B}||\mathcal{C}||\mathcal{D}|}.$$
\end{lemma}

The proof of Lemma \ref{baloglemma} uses only elementary geometry and it is similar in spirit to the proof of Theorem \ref{main1} and the earlier work of Solymosi. Later on, we will seek to prove a version of Theorem \ref{main2} which holds for the complex numbers, and the most difficult step in this process is to extend Lemma \ref{baloglemma} to the complex case. Full details of how this generalisation works will be given in Section 3.

\begin{theorem} \label{main2} Let $A$ be a finite set of positive real numbers and let $k\geq{1}$ be an integer. Then
\begin{equation}
|4^{k-1}A^{(k)}|\geq{|A|^k}.
\label{target}
\end{equation}
\end{theorem}

\begin{proof} The proof is by induction on $k$. In the base case when $k=1$, this is just the trivial statement that $|A|\geq{|A|}$. Now suppose that \eqref{target} holds for $k-1$; that is
\begin{equation}
|4^{k-2}A^{(k-1)}|\geq{|A|^{k-1}}.
\label{ih}
\end{equation}
Apply Lemma \ref{baloglemma} with $\mathcal{A}=\mathcal{B}=A+A$ and $\mathcal{C}=\mathcal{D}=4^{k-2}A^{(k-1)}$. By the inductive hypothesis \eqref{ih}, we obtain
\begin{equation}
|(A+A)4^{k-2}A^{(k-1)}+(A+A)4^{k-2}A^{(k-1)}|^2\geq{\left|\frac{A+A}{A+A}\right||A|^{2k-2}}.
\label{lb1}
\end{equation}
A crude application of Theorem \ref{main1} gives
\begin{equation}
\left|\frac{A+A}{A+A}\right|\geq{2|A|^2-1}\geq{|A|^2}.
\label{crude}
\end{equation}
Applying \eqref{crude}, it follows that
\begin{equation}
|(A+A)4^{k-2}A^{(k-1)}+(A+A)4^{k-2}A^{(k-1)}|^2\geq{|A|^{2k}}.
\label{lb2}
\end{equation}
On the other hand
\begin{align*}
(A+A)4^{k-2}A^{(k-1)}+(A+A)4^{k-2}A^{(k-1)}&\subseteq{4^{k-2}A^{(k)}+4^{k-2}A^{(k)}+4^{k-2}A^{(k)}+4^{k-2}A^{(k)}}
\\&=4^{k-1}A^{(k)},
\end{align*}
and this can be combined with \eqref{lb2}, in order to conclude that
$$|4^{k-1}A^{(k)}|\geq{|A|^k}.$$
\end{proof}
Theorem \ref{main2} is tight in the sense that the right hand side of \eqref{target} can only be improved by a constant factor which may depend on $k$. On the other hand, it seems likely that this bound could be improved significantly on the left hand side, and that there is scope for the value $4^{k-1}$ to be replaced by something much smaller. To this end, we make the following conjecture:
\begin{conjecture} Let $A$ be a finite set of real numbers and let $k\geq{1}$ be an integer. Then
$$|kA^{(k)}|\gg{|A|^k}.$$
\end{conjecture}

To conclude this section, let us observe another application of Lemma \ref{baloglemma}.
\begin{corollary}
For any finite set $A$ of real numbers,
$$|(A+A)(A+A)(A+A)+(A+A)(A+A)(A+A)|\gg{\frac{|A|^3}{\log{|A|}}}.$$
\end{corollary}
\begin{proof} Apply Lemma \ref{baloglemma} with $\mathcal{A}=\mathcal{B}=A+A$ and $\mathcal{C}=\mathcal{D}=(A+A)(A+A)$. This yields
\begin{equation}
|(A+A)(A+A)(A+A)+(A+A)(A+A)(A+A)|\gg{\left|\frac{A+A}{A+A}\right|^{1/2}|(A+A)(A+A)|}.
\label{lb3}
\end{equation}
As mentioned in the introduction, it was proven in \cite{rectangles} that
\begin{equation}
|(A+A)(A+A)|\gg{\frac{|A|^2}{\log{|A|}}}.
\label{misha}
\end{equation}
Theorem \ref{main1} and bound \eqref{misha} can be applied to bound the right hand side of \eqref{lb3} from below. It follows that
$$|(A+A)(A+A)(A+A)+(A+A)(A+A)(A+A)|\gg{\frac{|A|^3}{\log{|A|}}},$$
as required.
\end{proof}

This bound is tight up to constant and logarithmic factors. Indeed, it can be viewed as a weaker version of Theorem \ref{main2} in the case when $k=3$. Once again, it is the case when $A=\{1,\cdots,N\}$ which shows that this bound is close to being tight. It seems likely that a stronger estimate could be obtained using less variables. To be more precise, we make the following conjecture:

\begin{conjecture} Let $A\subset{\mathbb{R}}$ be a finite set. Then, for all $\epsilon>0$,
$$|(A+A)(A+A)(A+A)|\gg{|A|^{3-\epsilon}}.$$
\end{conjecture}

\section{The Complex Case}
In this section, we seek to extend Theorems \ref{main1} and \ref{main2} to the case when $A\subset{\mathbb{C}}$. Konyagin and Rudnev \cite{KR} generalised the result of Solymosi, from \cite{solymosi}, in order to prove that
\begin{equation}
|A+A|^2|A\cdot{A}|\gg{\frac{|A|^4}{\log{|A|}}},
\label{krmain}
\end{equation}
in this setting. In particular, they generalised the geometric setup of Solymosi. Recall that we used this setup in the proof of Theorem \ref{main1} to arrange the lines through the origin in increasing order of their gradient. Elementary geometry told us that the vector sums of point from neighbouring lines would lie in between the two lines, which allowed us to sum along neighbouring lines without overcounting the quantity we were interested in. An ingenious argument from \cite{KR} gives a suitable analogy for this construction over $\mathbb{C}$. The essential step is the following claim, which we repeat here using the same notation as in \cite{KR}:

\begin{claim} \label{mainclaim1}
Let $A$ be a subset of $\mathbb{C}\setminus{\{0\}}$ which is contained inside an angular sector $S:=\{z\in{\mathbb{C}}:|arg(z)|<\epsilon\}$, where $\epsilon>0$ is a fixed value which is sufficiently small for the argument to work ($\epsilon$ does not depend on $A$), and let $l_1$ and $l_2$ be two distinct elements of the ratio set $A/A\subset{\mathbb{C}}$. Fix representations $l_1=\frac{p_2}{p_1}$ and $l_2=\frac{q_2}{q_1}$, where $p_1,p_2,q_1,q_2\in{A}$. We can view $l_1$ and $l_2$ as points in $\mathbb{R}^2$. Then, the point $z=\frac{p_2+q_2}{p_1+q_1}$ lies in an open set $M_{(l_1,l_2)}$, which is symmetric about the open line interval $(l_1,l_2)$ connecting $l_1$ and $l_2$ and also contains $(l_1,l_2)$. 

Furthermore, the ratio set $A/A$ can be viewed as the vertex set in $\mathbb{R}^2$ for a geometric graph. Let $T$ to be a minimal spanning tree for these vertices; that is $T$ is a spanning tree on $A/A$ with the property that there is no spanning tree $T'$ such that
$$\sum_{(l_1,l_2)\in{E(T')}}|l_1-l_2|<\sum_{(l_1,l_2)\in{E(T)}}|l_1-l_2|,$$
where $|l_1-l_2|$ denotes the Euclidean distance between the points $l_1$ and $l_2$. Then the sets $M_{(l_1,l_2)}$, such that $(l_1,l_2)\in{E(T)}$, are pairwise disjoint.
\end{claim}

The sets $M_{(l_1,l_2)}$ are constructed explicitly in the proof of Claim \ref{mainclaim1}; we will also give an explicit construction of these sets as part of the proof of the forthcoming Lemma \ref{baloglemma2}. To illustrate how this claim generalises Solymosi's geometric argument, it is pointed out in \cite{KR} that if $A$ is a set of positive reals, then the spanning tree $T$ is just a straight path along the real axis, and the sets $M_{(l_1,l_2)}$ correspond to open line segments between neighbouring points in $A$.

Along with the proof of Theorem \ref{main1}, the proof of Lemma \ref{baloglemma} is based upon the geometric construction of Solymosi, and so the work of Konyagin and Rudnev provides a convenient framework for which to extend our results to the complex setting. We start this task by proving the folllowing Theorem:

\begin{theorem} \label{main3}
If $A\subset{\mathbb{C}}$ is finite then
$$\left|\frac{A+A}{A+A}\right|\gg{|A|^2}.$$
\end{theorem}

\begin{proof} In order to apply Claim \ref{mainclaim1}, we need to ensure that all of the arguments of elements of $A$ are sufficiently small. To do this, we just use the pigeonhole principle to identify a subset $A'\subset{A}$ such that $|A'|\gg{|A|}$ and for all $a_1,a_2\in{A'}$ we have $|arg(a_1)-arg(a_2)|<\epsilon$. Then, we rotate this set (that is, we dilate by a unit complex number $u$) to get a set $uA'$ with the required property that $|arg(z)|<\epsilon/2$ for all $z\in{uA'}$. Since the statement of the Theorem is invariant under dilation, it will suffice to show that
$$\left|\frac{uA'+uA'}{uA'+uA'}\right|\gg{|A'|^2},$$
as it then follows that
$$\left|\frac{A'+A'}{A'+A'}\right|=\left|\frac{uA'+uA'}{uA'+uA'}\right|\gg{|A'|^2}\gg{|A|^2}.$$
In order to simplify the notation slightly, we will assume from the outset that $A$ satisfies the conditions of Claim \ref{mainclaim1}. If this is not the case, then we relabel so that $uA'=A$, and the preceding analysis shows that we are dealing with the same problem.

Applying Claim \ref{mainclaim1}, we have a spanning tree $T$, whose edge set $E(T)$ consists of collection of $|A/A|-1$ pairs of elements of the ratio set. Write
$$E(T)=\{(l_i,m_i):1\leq{i}\leq{|A/A|-1}\}.$$
Since the sets $M_{(l_i,m_i)}$ such that $(l_i,m_i)\in{E(T)}$ are pairwise disjoint, it follows that
\begin{equation}
\left|\frac{A+A}{A+A}\right|\geq{\left|\bigcup_{i=1}^{|A/A|-1}\left(\frac{A+A}{A+A}\cap{M_{(l_i,m_i)}}\right)\right|}=\sum_{i=1}^{|A/A|-1}\left|\frac{A+A}{A+A}\cap{M_{(l_i,m_i)}}\right|.
\label{disjoint}
\end{equation}

Next, fix $i$, and consider the quantity $\left|\frac{A+A}{A+A}\cap{M_{(l_i,m_i)}}\right|$. There can be no loops in the minimal spanning tree $T$, and so certainly 
\begin{equation}
l_i\neq{m_i}.
\label{neq}
\end{equation}

Let $p=(p_1,p_2)$ be an arbitrarily chosen element of $A\times{A}$ with the property that $\frac{p_2}{p_1}=l_i$. Let $q=(q_1,q_2)$ and $q'=(q_1',q_2')$ be distinct elements of $A\times{A}$ such that 
\begin{equation}
\frac{q_2}{q_1}=\frac{q_2'}{q_1'}=m_i.
\label{eq2}
\end{equation}

\begin{claim} \label{claim1} 
$$\frac{p_2+q_2}{p_1+q_1}\neq{\frac{p_2+q_2'}{p_1+q_1'}}.$$
\end{claim}

\begin{proof} Suppose for a contradiction that
$$\frac{p_2+q_2}{p_1+q_1}=\frac{p_2+q_2'}{p_1+q_1'}.$$
After rearranging this expression, it follows that
\begin{equation}
p_2q_1'+q_2p_1+q_2q_1'=p_1q_2'+q_1p_2+q_1q_2'.
\label{long}
\end{equation}
By \eqref{eq2}, we have $q_2q_1'=q_2'q_1$, and therefore \eqref{long} can be simplified as follows:
\begin{equation}
p_2q_1'+q_2p_1=p_1q_2'+q_1p_2,
\label{short}
\end{equation}
and hence
$$p_1(q_2-q_2')=p_2(q_1-q_1').$$
It follows, from the fact that $q$ and $q'$ are distinct and $\frac{q_2}{q_1}=\frac{q_2'}{q_1'}$, that $q_1\neq{q_1'}$ and $q_2\neq{q_2'}$. Therefore,
\begin{equation}
\frac{p_2}{p_1}=\frac{q_2-q_2'}{q_1-q_1'}.
\label{long2}
\end{equation}
On the other hand, one can check that it follows from \eqref{eq2} that
$$\frac{q_2-q_2'}{q_1-q_1'}=\frac{q_2}{q_1},$$
and so
$$\frac{p_2}{p_1}=\frac{q_2}{q_1}.$$
However, this contradicts the fact that, by \eqref{neq},
$$\frac{p_2}{p_1}=l_i\neq{m_i}=\frac{q_2}{q_1}.$$
\end{proof}

Because of Claim \ref{claim1}, it is now known that, for this fixed element $p$ with $\frac{p_2}{p_1}=l_i$, we have
$$\left|\left\{\frac{p_2+q_2}{p_1+q_1}:(q_1,q_2)\in{A\times{A}},\frac{q_2}{q_1}=m_i\right\}\right|=\left|\left\{(q_1,q_2)\in{A\times{A}}:\frac{q_2}{q_1}=m_i\right\}\right|.$$
Also,
$$\left\{\frac{p_2+q_2}{p_1+q_1}:(q_1,q_2)\in{A\times{A}},\frac{q_2}{q_1}=m_i\right\}\subset{\frac{A+A}{A+A}\cap{M_{(l_i,m_i)}}},$$
and therefore
\begin{equation}
\left|\frac{A+A}{A+A}\cap{M_{(l_i,m_i)}}\right|\geq{\left|\left\{(q_1,q_2)\in{A\times{A}}:\frac{q_2}{q_1}=m_i\right\}\right|}.
\label{multiplicity1}
\end{equation}
Similarly, if we fix a specific point $q=(q_1,q_2)\in{A\times{A}}$ such that $\frac{q_2}{q_1}=m_i$, then the same argument shows that
\begin{align}
\left|\frac{A+A}{A+A}\cap{M_{(l_i,m_i)}}\right|&\geq{\left|\left\{\frac{p_2+q_2}{p_1+q_1}:(p_1,p_2)\in{A\times{A}},\frac{p_2}{p_1}=l_i\right\}\right|}
\\&=\left|\left\{(p_1,p_2)\in{A\times{A}}:\frac{p_2}{p_1}=l_i\right\}\right|.  \label{multiplicity2}
\end{align}
We abbreviate by introducing the notation 
$$r_{A/A}(x):=\left|\left\{(a,b)\in{A\times{A}}:\frac{b}{a}=x\right\}\right|,$$
for the number of representations of $x$ as an element of the ratio set $A/A$. Using this notation, it follows from \eqref{multiplicity1} and \eqref{multiplicity2} that
\begin{align}
\left|\frac{A+A}{A+A}\cap{M_{(l_i,m_i)}}\right|&\geq{\max{\{r_{A/A}(l_i),r_{A/A}(m_i)\}}}
\\&\geq{\frac{1}{2}\left(r_{A/A}(l_i)+r_{A/A}(m_i)\right)}
\\&\gg{r_{A/A}(l_i)+r_{A/A}(m_i)}. \label{close}
\end{align}
Now, since \eqref{close} holds for any $i$ in the range $1\leq{i}\leq{|A/A|-1}$, it follows that
\begin{equation}
\sum_{i=1}^{|A/A|-1}\left|\frac{A+A}{A+A}\cap{M_{(l_i,m_i)}}\right|\gg{\sum_{i=1}^{|A/A|-1}\big(r_{A/A}(l_i)+r_{A/A}(m_i)\big)}.
\label{closer}
\end{equation}
Since the edges of $T$ span the ratio set $A/A$, it follows that for every $x\in{A/A}$, the term $r_{A/A}(x)$ appears at least once in the right hand side of \eqref{closer}. Therefore,
\begin{equation}
\sum_{i=1}^{|A/A|-1}\left|\frac{A+A}{A+A}\cap{M_{(l_i,m_i)}}\right|\gg{\sum_{x\in{A/A}}r_{A/A}(x)}=|A|^2.
\label{closerer}
\end{equation}
Finally, \eqref{closerer} can be combined with \eqref{disjoint}, to conclude that
$$\left|\frac{A+A}{A+A}\right|\gg{|A|^2}.$$
\end{proof}

The next task is to generalise Lemma \ref{baloglemma}. This turns out to be a little bit more difficult, and it will be necessary to prove an adapted version of Claim \ref{mainclaim1}. Notice that the forthcoming Lemma \ref{baloglemma2} is slightly less general; we can only prove Lemma \ref{baloglemma} for complex numbers in the case when $\mathcal{C}=\mathcal{D}$, although this restriction is not a problem for all applications of the Lemma we have encountered.

\begin{lemma} \label{baloglemma2} Let $\mathcal{A,B}$ and $\mathcal{C}$ be finite sets of complex numbers. Then
$$|\mathcal{A}/\mathcal{B}||\mathcal{C}|^2\ll{|\mathcal{AC}+\mathcal{AC}||\mathcal{BC}+\mathcal{BC}|}.$$
\end{lemma}
\begin{proof} We may assume that $0\notin{\mathcal{A}\cup{\mathcal{B}}}$, since otherwise we simply delete this element from these sets, which effects only constant factors.

We begin with some pigeonholing, which is necessary in order to prove an adapted version of Claim \ref{mainclaim1}. There exists a small angular sector of the complex plane which contains a positive proportion of elements of $\mathcal{C}$. To be precise, let $\epsilon>0$ be a small but fixed value, which can be chosen with hindsight to be small enough for the proof to work. Then, there exists a set $\mathcal{C}'\subset{\mathcal{C}}$ such that $|\mathcal{C}'|\gg{|\mathcal{C}|}$ and
$$z,z'\in{\mathcal{C'}}\Rightarrow{|arg(z)-arg(z')|<\epsilon}.$$
Then, we can dilate (which can be thought of as a rotation of the complex plane) by some factor $\lambda_1\in{\mathbb{C}}$ so that all of the elements of $\lambda_1\mathcal{C}'$ are close to the real axis. By choosing $\lambda_1$ suitably, we have $|arg(z)|<\frac{\epsilon}{2}$ for all $z\in{\lambda_1\mathcal{C}'}$.

We also need to do some slightly more subtle pigeonholing on the ratio set $\mathcal{B}/\mathcal{A}$. First, choose a representative set $P$, by arbitrarily picking one pair from $\mathcal{A}\times{\mathcal{B}}$ to represent each element of $\mathcal{B}/\mathcal{A}$. To be precise, write $\mathcal{B}/\mathcal{A}=\{p_1,\cdots,p_{|\mathcal{B}/\mathcal{A}|}\}$, and for each $p_i\in{\mathcal{B}/\mathcal{A}}$ fix a representative pair $(a_i,b_i)$ such that $\frac{b_i}{a_i}=p_i$. We write
$$P:=\{(a_i,b_i):1\leq{i}\leq{|\mathcal{B}/\mathcal{A}|}\},$$
and note that
\begin{equation}
\frac{b_i}{a_i}=\frac{b_j}{a_j}\Leftrightarrow{i=j}.
\label{unique}
\end{equation}

Given $a\in{\mathcal{A}}$, define
$$w(a):=|\{b\in{\mathcal{B}}:(a,b)\in{P}\}|,$$
and note that
$$\sum_{a\in{\mathcal{A}}}w(a)=|P|=|\mathcal{B}/\mathcal{A}|.$$

By the pigeonhole principle, we can find a subset $\mathcal{A}'\subset{\mathcal{A}}$ which is contained in a small angular sector of $\mathbb{C}$ and makes a large contribution to the quantity $\sum_{a\in{\mathcal{A}}}w(a)$. To be precise, we assume for simplicity that $\frac{2\pi}{\epsilon}$ is an integer, and decompose the complex plane into disjoint slices
$$\mathbb{C}_j:=\{z\in{\mathbb{C}^*}:(j-1)\epsilon\leq{arg(z)}<j\epsilon\},$$
where $j\in{\mathbb{N}}$ ranges from $1$ to $\frac{2\pi}{\epsilon}$. Since these sets are disjoint, we have
$$\sum_{j=1}^{\frac{2\pi}{\epsilon}}\sum_{a\in{\mathcal{A}\cap{\mathbb{C}_j}}}w(a)=\sum_{a\in{\mathcal{A}}}w(a)=|P|.$$
Therefore, there exists an integer $j_0$ which determines a subset $\mathcal{A}':=\mathcal{A}\cap{\mathbb{C}_{j_0}}$, with the property that
$$\sum_{a\in{\mathcal{A}'}}w(a)\geq{\frac{\epsilon|\mathcal{B}/\mathcal{A}|}{2\pi}}\gg{|\mathcal{B}/\mathcal{A}|},$$
and $|arg(z)-arg(z')|<\epsilon$ for all $z,z'\in{\mathcal{A}'}$.
Once again, there exists a scalar $\lambda_2\in{\mathbb{C}}$ such that $|arg(z)|<\frac{\epsilon}{2}$ for all $z\in{\lambda_2\mathcal{A}'}$.

Define $P'\subset{P}$ to be the set
$$P':=\{(a,b)\in{P}:a\in{\mathcal{A}'}\},$$
and observe that
$$|P'|=\sum_{a\in{\mathcal{A}'}}|\{b\in{\mathcal{B}}:(a,b)\in{P}\}|=\sum_{a\in{\mathcal{A}'}}w(a)\gg{|\mathcal{B}/\mathcal{A}|}.$$
Next, we use these carefully chosen subsets $P'$ and $\mathcal{C}'$ to prove an adaptation of Claim \ref{mainclaim1}.

\begin{claim} \label{mainclaim2} For each $(a_i,b_i)\in{P'}$, we write $l_i=\frac{b_i}{a_i}$. Consider the set of ratios ${\{l_i:(a_i,b_i)\in{P'}\}}$, and fix a pair $(l_i,l_j)$ of elements from this set such that $i\neq{j}$. Then, for any pair $(c_1,c_2)\in{\mathcal{C}'\times{\mathcal{C'}}}$, the sum
$$(a_ic_1,b_ic_1)+(a_jc_2,b_jc_2)=(a_ic_1+a_jc_2,b_ic_1+b_jc_2)$$
determines a ratio
$$z=\frac{b_ic_1+b_jc_2}{a_ic_1+a_jc_2}\in{\mathbb{C}},$$
and this $z$ lies in an open set $M_{(l_i,l_j)}$, which is symmetric about the open line interval $(l_i,l_j)$ connecting $l_i$ and $l_j$ and also contains $(l_i,l_j)$.

Furthermore, the set of ratios $\{l_i:(a_i,b_i)\in{P'}\}$ can be viewed as the vertex set in $\mathbb{R}^2$ for a geometric graph. Let $T$ to be a minimal spanning tree for these vertices; that is $T$ is a spanning tree for $\{l_i:(a_i,b_i)\in{P'}\}$ with the property that there is no spanning tree $T'$ such that
$$\sum_{(l_i,l_j)\in{E(T')}}|l_i-l_j|<\sum_{(l_i,l_j)\in{E(T)}}|l_i-l_j|,$$
where $|l_i-l_j|$ denotes the Euclidean distance between the points $l_i$ and $l_j$. Then the sets $M_{(l_i,l_j)}$, such that $(l_i,l_j)\in{E(T)}$, are pairwise disjoint.
\end{claim}
\begin{proof} Write
$$u:=\frac{a_jc_2}{a_ic_1}=\frac{(\lambda_2a_j)(\lambda_1c_2)}{(\lambda_2a_i)(\lambda_1c_1)}.$$
One can check that
\begin{equation}
z=\frac{b_ic_1+b_jc_2}{a_ic_1+a_jc_2}=l_i+(l_j-l_i)\frac{u}{1+u}.
\label{important}
\end{equation}
By the earlier pigeonholing, we know that ${|arg(\lambda_2a_j)|,|arg(\lambda_2a_i)|,|arg(\lambda_1c_2)|,|arg(\lambda_1c_1)|<\frac{\epsilon}{2}}$. It therefore follows that $|arg(u)|<2\epsilon$, that is
$$u\in{W_{2\epsilon}}:=\{z:arg(z)<2\epsilon\}.$$
Therefore, $\frac{u}{1+u}$ lies in the set $M(W_{2\epsilon}):=M_{2\epsilon}$, which is the image of $W_{2\epsilon}$ under the M\"{o}bius map $M(z)=\frac{z}{1+z}$.

Define $M_{(l_i,l_j)}$ from the statement of the claim to be the set
$$M_{(l_i,l_j)}:=\{l_i+(l_j-l_i)M_{2\epsilon}\}.$$
By \eqref{important}, we have
$$z=\frac{b_ic_1+b_jc_2}{a_ic_1+a_jc_2}\in{M_{(l_i,l_j)}}.$$
The rest of the proof is concerned with establishing that the sets $M_{(l_i,l_j)}$ are symmetric with respect to the open straight line from $l_i$ to $l_j$ whilst also containing the line, and also that the sets are pairwise disjoint. This is exactly what has already been established in the proof of the original claim in \cite{KR}, and so we refer the reader to \cite{KR} for a full account.
\end{proof}

We now return to the proof of Lemma \ref{baloglemma2}. Fix an edge $(l_i,l_j)\in{E(T)}$, the edge set of our minimal spanning tree. Consider the set of sums
\begin{align*}S_{ij}:&=\{(a_ic_1,b_ic_1)+(a_jc_2+b_jc_2):c_1,c_2\in{\mathcal{C}'}\}
\\&=\{(a_ic_1+a_jc_2,b_ic_1+b_jc_2):c_1,c_2\in{\mathcal{C}'}\},
\end{align*}
and observe that $S_{ij}\subset{(\mathcal{AC}+\mathcal{AC})\times{(\mathcal{BC}+\mathcal{BC})}}$. The proof of the lemma will follow easily once we have established the following two claims:

\begin{claim} \label{firstclaim}
For any $(l_i,l_j)\in{E(T)}$, we have $|S_{ij}|=|\mathcal{C}'|^2$.
\end{claim}

\begin{claim} \label{secondclaim}
For any pair of distinct edges $(l_i,l_j),(l_{i'},l_{j'})\in{E(T)}$, we have $S_{ij}\cap{S_{i'j'}}=\emptyset$.
\end{claim}

First, let us show how these claims can be used to conclude the proof of Lemma \ref{baloglemma2}. Note that
$$\bigcup_{(l_i,l_j)\in{E(T)}}S_{ij}\subset{(\mathcal{AC}+\mathcal{AC})\times{(\mathcal{BC}+\mathcal{BC})}}.$$
By Claim \ref{secondclaim}, this is a disjoint union, and so
\begin{align}
|\mathcal{AC}+\mathcal{AC}||\mathcal{BC}+\mathcal{BC}|&=|(\mathcal{AC}+\mathcal{AC})\times{(\mathcal{BC}+\mathcal{BC})}|
\\&\geq{\sum_{(l_i,l_j)\in{E(T)}}|S_{ij}|}
\\&=\sum_{(l_i,l_j)\in{E(T)}}|\mathcal{C}'|^2 \label{ap1}
\\&\gg{\sum_{(l_i,l_j)\in{E(T)}}|\mathcal{C}|^2}, \label{ap2}
\end{align}
where \eqref{ap1} is a consequence of Claim \ref{firstclaim}, and \eqref{ap2} comes from the construction of the set $\mathcal{C}'$.

Now, since $T$ is a minimal spanning tree for the vertex set $\{l_i:(a_i,b_i)\in{P'}\}$, it must be the case that
$$|E(T)|=|P'|-1\gg{|\mathcal{B}/\mathcal{A}|}.$$
Combining this with \eqref{ap2}, we have
$$|\mathcal{AC}+\mathcal{AC}||\mathcal{BC}+\mathcal{BC}|\gg{|\mathcal{B}/\mathcal{A}||\mathcal{C}|^2},$$
as required.

It remains to prove Claims \ref{firstclaim} and \ref{secondclaim}.

\textit{Proof of Claim \ref{firstclaim}}. To check that $|S_{ij}|=|\mathcal{C}'|^2$, we need to check that there is no repetition in $S_{ij}$. That is, we need to show that there exist no non-trivial solutions to the equation
\begin{equation}
(a_ic_1+a_jc_2,b_ic_1+b_jc_2)=(a_ic_3+a_jc_4,b_ic_3+b_jc_4),
\label{nontriv}
\end{equation}
such that $(c_1,c_2,c_3,c_4)\in{\mathcal{C}'^4}$. By a non-trivial solution, we mean a solution such that $(c_1,c_2)\neq{(c_3,c_4)}$. Indeed, suppose that a quadruple $(c_1,c_2,c_3,c_4)\in{\mathcal{C}'^4}$ satisfies \eqref{nontriv}, and let us show that it must be the case that $(c_1,c_2)=(c_3,c_4)$. We have
$$a_i(c_1-c_3)+a_j(c_2-c_4)=0$$
and
$$b_i(c_1-c_3)+b_j(c_2-c_4)=0.$$
Let us assume for a contradiction that $c_2-c_4\neq{0}$. Then, recalling that $0\notin{\mathcal{A,B}}$, we have
$$\frac{-a_j}{a_i}=\frac{c_1-c_3}{c_2-c_4}.$$
and
$$\frac{-b_j}{b_i}=\frac{c_1-c_3}{c_2-c_4}.$$
This tells us that $\frac{b_i}{a_i}=\frac{b_j}{a_j}$. However, \eqref{unique} then implies that $i=j$. This is a contradiction, since we then have a loop edge $(l_i,l_i)\in{E(T)}$, but the minimal spanning tree $T$ cannot contain any loops. Therefore, our assumption that $c_2\neq{c_4}$ must be false. An identical argument then shows that $c_1=c_3$, and this confirms that there exist no non-trivial solutions to \eqref{nontriv}, as required.
\begin{flushright}
\qedsymbol
\end{flushright}

\textit{Proof of Claim \ref{secondclaim}}. Suppose for a contradiction that there exists a pair $(z_1,z_2)\in{S_{ij}\cap{S_{i'j'}}}$. Then, since $(z_1,z_2)\in{S_{ij}}$, it can be written in the form
$$(z_1,z_2)=(a_ic_1+a_jc_2,b_ic_1+b_jc_2)$$
for some $c_1,c_2\in{\mathcal{C}'}$. Then, Claim \ref{mainclaim2} tells us that
$$\frac{z_2}{z_1}\in{M_{(l_i,l_j)}}.$$
However, the same argument then shows that
$$\frac{z_2}{z_1}\in{M_{(l_{i'},l_{j'})}}.$$
Therefore, $z\in{M_{(l_i,l_j)}\cap{M_{(l_{i'},l_{j'})}}}$, where $(l_i,l_j)$ and $(l_i',l_j')$ are distinct edges of $T$. This contradicts the non-intersection property established in Claim \ref{mainclaim2}. Our original assumption must therefore be false, and so
$$S_{ij}\cap{S_{i'j'}}=\emptyset,$$
as required.
\begin{flushright}
\qedsymbol
\end{flushright}

Having proven Claims \ref{firstclaim} and \ref{secondclaim}, the proof of Lemma \ref{baloglemma2} is now comnplete.
\end{proof}

It is now a straightforward task to prove a version of Theorem \ref{main2} in the complex setting as follows:

\begin{theorem} \label{main4} Let $A$ be a finite set of complex numbers and let $k\geq{1}$ be an integer. Then
\begin{equation}
|4^{k-1}A^{(k)}|\gg_{k}{|A|^k}.
\label{target2}
\end{equation}
\end{theorem}

\begin{proof} Theorem \ref{main3} implies that there exist a fixed constant $c_1>0$ such that
\begin{equation}
\left|\frac{A+A}{A+A}\right|\geq{c_1|A|^2}.
\label{c1}
\end{equation}
Similarly, by Lemma \ref{baloglemma2}, there exists $c_2>0$ such that
\begin{equation}
|\mathcal{AC}+\mathcal{AC}|^2\geq{c_2|\mathcal{C}|^2|\mathcal{A}/\mathcal{A}|},
\label{c2}
\end{equation}
for any finite sets $\mathcal{A},\mathcal{C}\subset{\mathbb{C}}$.

To prove Theorem \ref{main4}, it will be established that
\begin{equation}
|4^{k-1}A^{(k)}|\geq{(c_1c_2)^{\frac{k-1}{2}}|A|^k}.
\label{aim}
\end{equation}

The proof is by induction on $k$. The base case when $k=1$ is just the trivial statement that $|A|\geq{|A|}$. Now suppose that \eqref{aim} holds for $k-1$; that is
\begin{equation}
|4^{k-2}A^{(k-1)}|\geq{(c_1c_2)^{\frac{k-2}{2}}|A|^{k-1}}.
\label{ih2}
\end{equation}
Apply \eqref{c2} with $\mathcal{A}=A+A$ and $\mathcal{C}=4^{k-2}A^{(k-1)}$. By the inductive hypothesis \eqref{ih2}, and \eqref{c1}, we obtain
\begin{align}
|(A+A)4^{k-2}A^{(k-1)}+(A+A)4^{k-2}A^{(k-1)}|^2&\geq{c_2(c_1c_2)^{k-2}|A|^{2k-2}\left|\frac{A+A}{A+A}\right|}
\\&\geq{(c_2c_1)^{k-1}|A|^{2k}}. \label{lb12}
\end{align}
On the other hand
\begin{align*}
(A+A)4^{k-2}A^{(k-1)}+(A+A)4^{k-2}A^{(k-1)}&\subseteq{4^{k-2}A^{(k)}+4^{k-2}A^{(k)}+4^{k-2}A^{(k)}+4^{k-2}A^{(k)}}
\\&=4^{k-1}A^{(k)},
\end{align*}
and this can be combined with \eqref{lb12}, in order to conclude that
$$|4^{k-1}A^{(k)}|\geq{(c_1c_2)^{\frac{k-1}{2}}|A|^k},$$
as required
\end{proof}

\section*{Acknowledgements}
The authors are very grateful to Misha Rudnev for his suggestion of extending these results to the complex setting.

\end{document}